\documentclass[11pt]{article}
\usepackage{amsmath, amssymb, amsthm}
\usepackage{verbatim}
\usepackage{multicol}
\usepackage{enumerate}
\usepackage{comment}
\usepackage{dsfont}
\usepackage[none]{hyphenat}
\usepackage[unicode]{hyperref}
\hypersetup{
	colorlinks=true,
	linkcolor=blue,
	filecolor=magenta,      
	urlcolor=cyan,
	citecolor=blue
}
\usepackage{pgf}
\usepackage{tikz}
\usetikzlibrary{positioning,arrows,shapes,decorations.markings,decorations.pathreplacing,matrix,patterns}
\tikzstyle{vertex}=[circle,draw=black,fill=black,inner sep=0,minimum size=3pt,text=white,font=\footnotesize]
\usepackage{cleveref}

\date{}
\title{\vspace{-1.2cm} Bisection Width, Discrepancy, and Eigenvalues of Hypergraphs}

\author{Eero R\"aty\thanks{Ume\r{a} University, \emph{e-mail}: \textbf{eero.raty@umu.se}. Supported by a postdoctoral grant from the Osk.\ Huttunen Foundation.}, 
	Istv\'an Tomon\thanks{Ume\r{a} University, \emph{e-mail}: \textbf{istvan.tomon@umu.se}. Supported by the Swedish Research Council grant 2023-03375.}
}

\theoremstyle{plain}
\newtheorem{theorem}{Theorem}[section]

\newtheorem{claim}[theorem]{Claim}
\newtheorem{lemma}[theorem]{Lemma}

\newtheorem*{problem*}{Problem}

\Crefname{theorem}{Theorem}{Theorems}
\Crefname{definition}{Definition}{Definitions}
\Crefname{corollary}{Corollary}{Corollaries}
\Crefname{claim}{Claim}{Claims}
\Crefname{lemma}{Lemma}{Lemmas}
\Crefname{conjecture}{Conjecture}{Conjectures}
\Crefname{problem}{Problem}{Problems}
\Crefname{prop}{Proposition}{Propositions}

\theoremstyle{definition}
\newtheorem{definition}{Definition}

\DeclareMathOperator{\disc}{disc}

\DeclareMathOperator{\vol}{vol}

\DeclareMathOperator{\bw}{bw}

\begin{document}
	
	\maketitle
	\sloppy

\begin{abstract}
A celebrated result of Alon from 1993 states that any $d$-regular graph on $n$ vertices (where $d=O(n^{1/9})$) has a bisection with at most $\frac{dn}{2}(\frac{1}{2}-\Omega(\frac{1}{\sqrt{d}}))$ edges, and this is optimal. Recently, this result was greatly extended by R\"aty, Sudakov, and Tomon. We build on the ideas of the latter, and use a semidefinite programming inspired approach to prove the following variant for hypergraphs: every $r$-uniform $d$-regular hypergraph on $n$ vertices (where $d\ll n^{1/2}$) has a bisection of size at most $$\frac{dn}{r}\left(1-\frac{1}{2^{r-1}}-\frac{c}{\sqrt{d}}\right),$$
for some $c=c(r)>0$. This bound is the best possible up to the precise value of $c$. Moreover, a bisection achieving this bound can be found by a polynomial-time randomized algorithm.

The  minimum bisection is closely related to discrepancy. We also prove sharp bounds on the discrepancy and so called positive discrepancy of hypergraphs, extending results of Bollob\'as and Scott. Furthermore, we discuss implications about Alon-Boppana type bounds. We show that if $H$ is an $r$-uniform $d$-regular hypergraph, then certain notions of second largest eigenvalue $\lambda_2$ associated with the adjacency tensor satisfy  $\lambda_2\geq \Omega_r(\sqrt{d})$, improving results of Li and Mohar.
\end{abstract}
 
\section{Introduction}

\subsection{Bisection width}

An \emph{equipartition} of a finite set is a partition into two parts, whose sizes differ by at most one. A  \emph{bisection} of a graph $G$ is an equipartition of its vertex set, together with all the edges containing one vertex in each part, and the \emph{size} of a bisection is the number of its edges. The \emph{bisection width} or \emph{minimum bisection}  of a graph $G$ is the minimum size of a bisection, and it is denoted by $\bw(G)$. Due to its importance in theoretical computer science, the algorithmic aspect of the bisection width problem has been studied extensively. For the background and recent advances in the area, see e.g.\ \cite{HLRW, HRW, Karger, KT, Li}

A fundamental result in this topic is due to Alon \cite{Alon93} (see also \cite{AHK99}), which states that every $d$-regular graph on $n$ vertices has bisection width at most $\frac{nd}{4}- \Omega(\sqrt{d}n)$ whenever $d = O\left(n^{1/9}\right)$. On the other hand, Bollob\'{a}s \cite{BB88} proved that the bisection width of almost all $d$-regular graph is at least $\frac{nd}{4} - \frac{n\sqrt{d \ln 2}}{2}$, thus matching the bound of Alon \cite{Alon93} up to a constant factor in the error term. The problem of improving these bounds for random $d$-regular graphs has been a line of active research in the interface of combinatorics and theoretical computer science, for example due to its connections with finding internal partitions. Due to the difficulty of the problem, a specific point of focus has been on small constant values of $d$, see e.g.\ the works  \cite{BNP, DDSW, DSW, KM, MP}.

Recently, R\"aty, Sudakov, and Tomon \cite{RST} substantially extended the result of Alon \cite{Alon93} in the range $d \leq (\frac{1}{2}-\varepsilon)n$. They proved that $\frac{nd}{4} - \Omega(\sqrt{d}n)$ remains an upper bound for the bisection width if $d = O\left(n^{2/3}\right)$, thus random graphs minimize this quantity asymptotically. However, surprisingly, in case $n^{2/3}\ll d< \left(\frac{1}{2}-\varepsilon\right)n$, the situation completely changes. When $n^{2/3}\leq d\leq n^{3/4}$, the maximum of the bisection width among $d$-regular graphs is   $\frac{nd}{4} - \Theta\left(n^2/d\right)$, and random graphs are beaten by certain families of strongly-regular graphs. The range $n^{3/4}\leq d\leq \left(\frac{1}{2}-\varepsilon\right)n$ is even more mysterious, we refer the interested reader to \cite{RST} for further details.

\medskip

The bisection width also extends to hypergraphs naturally as follows. The \emph{bisection} of a hypergraph $H$ is an equipartition of its vertex set, together with all the edges that contain at least one vertex in each part. The \emph{bisection width} of a hypergraph $H$ is the minimum number of edges in a bisection, and we again denote it by $\bw(H)$. The study of hypergraph bisection width has attracted attention from the algorithmic point of view \cite{FK,KWMY,RSS}. A particular topic of interest has been the hypergraph-generalisation of the $s$-$t$ cut problem. Here the aim is to find a minimum cut (not necessarily bisection) of $H$ that has two predetermined vertices $s$ and $t$ placed on the different sides of the cut \cite{CX}.

In this paper, our goal is to study the minimum bisection problem for hypergraphs from an extremal point of view, in particular to generalize the aforementioned results of Alon \cite{Alon93}, and R\"aty, Sudakov, and Tomon \cite{RST}. We remark that the case of 3-uniform hypergraphs can be reduced to multigraphs, as the size of a bisection in a 3-uniform hypergraph is equal to half the size of the corresponding bisection of the underlying multi-graph (i.e., the graph in which each edge is included as many times as it appears in a hyperedge). However, no similar reduction is possible if the uniformity is at least $4$.

Let us consider the minimum bisection of an  $r$-uniform hypergraph $H$ with $n$ vertices and average degree $d$. The expected size of a random bisection in $H$ is $$e(H)\cdot \left(1-\frac{1}{2^{r-1}}+\Theta_r\left(\frac{1}{n}\right)\right),$$
so this is always a trivial upper bound for $\bw(H)$. On the other hand, if $H$ is the random hypergraph in which every edge is included independently with probability $p=d/\binom{n-1}{r-1}\leq 1/2$, then the average degree of $H$ is $\approx d$ and
$$\bw(H)\geq e(H)\cdot \left(1-\frac{1}{2^{r-1}}\right)-O(\sqrt{d}n)$$
with high probability. See the next subsection for a detailed argument. As one of our main results, we establish an upper bound for the bisection width of $d$-regular hypergraphs that matches the previous lower bound, given $d$ is not too large with respect to $n$.

\begin{theorem}\label{thm:bisection}
Let $r\geq 2$ be an integer, then there exists $c,\varepsilon>0$ such that the following holds. Let $H$ be an $r$-uniform $d$-regular hypergraph on $n$ vertices, where $d \leq \varepsilon n^{1/2}$. Then 
$$\bw(H) \leq  e(H)\cdot\left(1 - \frac{1}{2^{r-1}}\right) - c\sqrt{d}n.$$
\end{theorem}

  Here, we remark that the same result is true if instead of regularity, we only assume that $H$ has average degree $d$ and maximum degree $O(d)$. Also, a bisection achieving this bound can be found with a polynomial-time randomized algorithm. However, the result no longer holds without some restriction on the maximum degree: if $G$ is the complete bipartite graph with vertex classes of size $d/2$ and $(n-d/2)$, then $G$ has average degree $d(1-o_n(1))$, but  $\bw(G)\geq \frac{1}{2}e(G)$. Furthermore, a similar result  no longer holds if $d\gg n^2$ (which also assumes $r\geq 4$). Indeed, if $H$ is the random hypergraph in which every edge is included independently with probability $p=d/\binom{n-1}{r-1}\leq 1/2$, then $\bw(H)=e(H)\cdot \left(1-\frac{1}{2^{r-1}}\right)+\Theta_r(d)$. See the next section for further discussion.

  Finally, we remark that Theorem \ref{thm:bisection} naturally extends to hypergraphs that are not necessarily uniform. If $H$ is such a hypergraph, then a random bisection has size at least $\sum_{e\in E(H)}1-2^{1-|e|}$. We show that if the maximum size of an edge is bounded by $r$, there is a bisection with significantly less edges.

  \begin{theorem}\label{thm:mixed}
Let $r\geq 2$ be an integer, then there exists $c,\varepsilon>0$ such that the following holds. Let $H$ be a  hypergraph on $n$ vertices with $m$ edges such that each edge is of size at most $r$, and the maximum degree is $\Delta$, where $\Delta^2 \leq \varepsilon mn^{1/2}$. Then 
$$\bw(H) \leq  \left(\sum_{e\in E(H)}1-2^{1-|e|}\right) - \frac{cm}{\sqrt{\Delta}}.$$
\end{theorem}
Note that this theorem immediately implies Theorem \ref{thm:bisection} by having $\Delta=d$.

\subsection{Discrepancy}

Let $H$ be an $r$-uniform hypergraph with $n$ vertices and edge density $p=\frac{|E(H)|}{\binom{n}{r}}$. Given $U\subset V(H)$, define the discrepancy of $U$ as 
$$\disc(U)=e(U)-p\binom{|U|}{r}.$$
Then $\disc(U)$ measures how much $U$ deviates from its expected size. The \emph{discrepancy} of $H$ is defined as the maximum absolute discrepancy over all subsets of vertices, that is,
$$\disc(H)=\max_{U\subset V(H)} |\disc(U)|.$$
 This notion of discrepancy was introduced by Erd\H{o}s, Goldbach, Pach and Spencer \cite{EGPS} in the 80's, extending earlier notions studied by Erd\H{o}s and Spencer \cite{ES71}. In \cite{EGPS}, it is proved that if $G$ is a graph on $n$ vertices and its edge density $p$ satisfies $\frac{1}{n}\leq p\leq \frac{1}{2}$, then $\disc(G)=\Omega(p^{1/2}n^{3/2})$, and equality is a achieved by the Erd\H{o}s-R\'enyi random graph $G_{n,p}$. When $p<1/n$, it is not difficult to show that the right lower bound is $\Omega(pn^2)$. Noting that the discrepancy of a hypergraph is equal to the discrepancy of its complement, these provide sharp bounds in case $\frac{1}{2}\leq p\leq 1$ as well. Bollob\'as and Scott \cite{BS06} extended this result to $r$-uniform hypergraphs in case $p$ is not too small. More precisely, they proved that if $\frac{1}{n}\leq p\leq \frac{1}{2}$, and $H$ is an $n$ vertex $r$-uniform hypergraph of edge density $p$, then $\disc(H)=\Omega_r(p^{1/2}n^{\frac{r+1}{2}})$. Again, equality is achieved by random hypergraphs. Here, we prove that the same bound holds for the whole range of interest $n^{-(r-1)}\ll p\leq \frac{1}{2}$

\begin{theorem}\label{thm:main_disc}
Let $H$ be an $r$-uniform hypergraph on $n$ vertices of average degree $d$, where $1\leq d\leq\frac{1}{2}\binom{n-1}{r-1}$. Then $$\disc(H)=\Omega_r(\sqrt{d}n).$$ 
\end{theorem}

 If $d<1$, it is easy to argue that the minimum discrepancy is $\Omega_r(dn)=\Omega_r(pn^{r})$. Indeed, in this case $H$ contains an independent set $U$ of size $\Omega(n)$, and $|\disc(U)|=p\binom{|U|}{r}=\Omega_r(pn^r)$.
 
 While the discrepancy of a hypergraph measures the maximum absolute deviation of the size of an induced subhypergraph compared to its expected size, it is also natural to consider whether this deviation is positive or negative. The \emph{positive discrepancy} of $H$ is defined as
$$\disc^{+}(H)=\max_{U\subset V(H)}\disc(U),$$ and the \emph{negative discrepancy} of $H$ is
$$\disc^{-}(H)=\max_{U\subset V(H)}-\disc(U).$$
Note that one trivially has $\disc^{+}(G) = \disc^{-}(G^c)$ and $\disc(G) = \max(\disc^{+}(G), \disc^{-}(G))$.

We observe that the positive discrepancy and the bisection width of a hypergraph are closely connected in a sense that small bisection width implies large positive discrepancy. 

\begin{lemma}\label{lemma:relating pos disc and bandwidth}
Let $H$ be an $r$-uniform hypergraph on $n$ vertices with average degree $d$. Let $s(H)$ be such that
$$\bw(H) = \frac{nd}{r}\left(1-2^{1-r}\right) - s(H).$$
Then
$$\disc^{+}(H) \geq \frac{s(H)}{2}.$$
\end{lemma}
\begin{proof}
Let $V(H) = X \cup Y$ be an equipartition such that the size of the corresponding bisection is $\bw(H)$. Note that for all $n \geq r$ we have
$$\binom{\lfloor\frac{n}{2}\rfloor}{r} + \binom{\lceil \frac{n}{2} \rceil}{r} \leq 2^{1-r}\binom{n}{r}.$$
Hence, it follows that 
$$\disc(X) + \disc(Y) = e(X) + e(Y) - p\left(\binom{\lfloor\frac{n}{2}\rfloor}{r} + \binom{\lceil \frac{n}{2} \rceil}{r}\right) \geq e(X) + e(Y) - 2^{1-r} \cdot p \binom{n}{r}.$$
Since $e(X) + e(Y) = e(H) - \bw(H)$, we get
$$\disc(X) + \disc(Y) \geq s(H).$$ 
In particular, we conclude that $\disc^{+}(H) \geq \frac{s(H)}{2}$.
\end{proof}

 In the case of graphs, exploring the connection between the bisection width  and positive discrepancy, R\"aty, Sudakov, and Tomon \cite{RST} proved sharp bounds on both of these quantities. However, in the case of hypergraphs, less is known. Bollob\'as and Scott \cite{BS06} proved that if $H$ is an $r$-uniform hypergraph with density $p$ satisfying $p(1-p) \geq \frac{1}{n}$, then $\disc^{+}(H)\cdot \disc^{-}(H) = \Omega_r(p(1-p)n^{r+1})$. Unfortunately, this inequality does not provide any information on $\disc^{+}(H)$ and $\disc^{-}(H)$ individually, beyond that both are at least $\Omega_r(n)$.

On the other hand, Bollob\'as and  Scott \cite{BS06} proved that for $d \geq 1$, the random $r$-uniform hypergraph $H$, where every edge is included independently with probability $p = d/\binom{n-1}{r-1}$ satisfies $\disc(H) = O(\sqrt{d}n)$ with high probability, which implies $\disc^{+}(H),\disc^{-}(H)=O(\sqrt{d}n)$ as well. Thus, by Lemma \ref{lemma:relating pos disc and bandwidth},  we also have $\bw(H) \geq (1- 2^{1-r})e(H) - O(\sqrt{d}n)$, confirming our claim in the previous section. Our main result concerning the positive discrepancy is the following lower bound for sufficiently sparse hypergraphs. We prove that whenever the average-degree $d$ is at most $n^{2/3}$, the positive discrepancy is bounded below by $\Omega(\sqrt{d}n)$, which is sharp by the aforementioned result.

\begin{theorem}\label{thm:main_pos_disc}
Let $H$ be an $r$-uniform hypergraph on $n$ vertices of average degree $d<n^{2/3}$. Then
$$\disc^{+}(H)=\Omega_r(d^{1/2}n).$$
\end{theorem}

\subsection{Eigenvalues}

 Given a $d$-regular graph $G$, let $A$ be the adjacency matrix of $G$ and let  $d=\lambda_1\geq \dots\geq \lambda_n$ be the eigenvalues of $A$. The Alon-Boppana theorem \cite{alon-boppana} is one of the central results in spectral graph theory, stating that the second largest eigenvalue satisfies $$\lambda_2\geq 2\sqrt{d-1}-o_n(1).$$ This is known to be tight for infinite values of $d$ due to the existence of so called Ramanujan graphs \cite{ramanujan}. Furthermore, as a celebrated result of Friedman \cite{Friedman} shows, $\lambda_2=2\sqrt{d-1}+o_n(1)$ for random $d$-regular graphs as well. Here, it is good to point out that $2\sqrt{d-1}$ also coincides with the spectral radius of the $d$-regular infinite tree. This might not be unexpected, as random $d$-regular graphs are locally tree-like. However, the exact connection between these two quantities might be more subtle, as we discuss later in this section. 


There has been a plethora of work concerning the spectral theory of hypergraphs under various frameworks \cite{B21,C93,CD,FL,FW,HQ,KLM,LM}. We now briefly discuss the main directions and some of the highlights of these works. Some authors \cite{B21,C93} define the adjacency matrix of a hypergraph $H$ as the matrix $A$ whose entry $A_{i,j}$ is the co-degree of the distinct vertices $i$ and $j$. In contrast, the rest of the literature considers the so called adjacency tensor. A tensor naturally corresponds to a multilinear map, so we rather define this map directly instead. We follow the notation of Li and Mohar \cite{LM}, which coincides with the notation of other sources up to constant factors depending on the uniformity.

\begin{definition}
    Given an $r$-uniform hypergraph $H$ on vertex set $V$, its \emph{adjacency map} $\tau_H:(\mathbb{C}^{V})^r\rightarrow \mathbb{R}$ is the symmetric multilinear function defined as follows: for every $x_1,\dots,x_r\in \mathbb{C}^V$,  
    $$\tau_H(x_1,\dots,x_r) = \frac{1}{(r-1)!}\sum_{\substack{v_1,\dots,v_r\in V\\ \{v_1,\dots,v_r\}\in E(H)}} x_1(v_1)\dots x_r(v_r).$$
    We define the normalized adjacency map of $H$ as $$\sigma_H = \tau_H - \frac{r|E(H)|}{n^r}J,$$ where $J$ is the ''all-one'' tensor defined as $J(x_1,\dots,x_r)=\sum_{v_1,\dots,v_r\in V} x_1(v_1)\dots x_r(v_r)$.
\end{definition}

Note that in case $G$ is a graph with adjacency matrix $A$, then $\tau_G(x,y)=x^T A y$. Also, if $\mathds{1}$ is the all-one vector, then $\sigma_H(\mathds{1},\dots,\mathds{1})=0$. We now define two sets of quantities that are potential candidates for the second largest eigenvalue of $r$-uniform hypergraphs.
\begin{definition}
    For $p>0$, the $L^p$-norm of a vector $x\in \mathbb{C}^n$ is defined as $$||x||_p=(|x(1)|^p+\dots+|x(n)|^p)^{1/p}.$$ Given an $r$-uniform hypergraph $H$, let
    $$\lambda_{2}^{(p)}(H) = \sup_{x\in \mathbb{R}^V,||x||_p= 1} \sigma_H(x,\dots,x),$$
    and
    $$\mu^{(p)}(H) = \sup_{||x_1||_p=1,||x_2||_p=1,\dots,||x_r||_p=1} |\sigma_H(x_1,\dots,x_r)|.$$
\end{definition}
Note that in case $r = 2$, $\lambda_2(H)=\lambda_2^{(2)}(H)$ and $\mu^{(2)}(H) = \max\left(|\lambda_2(H)|, \vert\lambda_n(H)\vert\right)$. Thus, in a sense, $\lambda_2^{(p)}$ is more related to the second largest eigenvalue of $H$, while $\mu^{(p)}$ is related to the second largest absolute value of the eigenvalues. Clearly, $\mu^{(p)}(H)\geq \lambda_2^{(p)}(H)$, but the gap between these quantities can be arbitrarily large: when $G$ is the complete balanced bipartite graph on $n$ vertices, $\lambda_2(G)=0$, while $\mu^{(2)}(G)=n/2$. Also, while $\mu^{(2)}(G)\geq \Omega(\sqrt{d})$ holds for every $d$-regular graph $G$ on $n\geq 2d$ vertices, and this bound is sharp, the minimum of $\lambda_2^{(2)}(G)$ among $d$-regular graphs has a much stranger behavior, see the recent manuscript \cite{RST}. In the case of $r$-uniform hypergraphs for $r\geq 3$, typically $\mu^{(2)}(H)$ or $\mu^{(r)}(H)$ is studied as the second eigenvalue in the literature, but we propose $\lambda_2^{(p)}(H)$ as a stronger alternative. 

A hypergraph $H$ is said to be \emph{$k$-co-degree regular} if for distinct $v_1,\dots,v_{r-1}$, the number of edges containing $\left\{v_1,\dots,v_{r-1}\right\}$ is $k$. Friedman and Wigderson \cite{FW}  studied the quantity $\mu^{(2)}(H)$ for co-degree regular hypergraphs, motivated by the theory of Cayley hypergraphs. They proved that if $H$ is a $3$-uniform   $k$-co-degree regular hypergraph on $n$ vertices, then $\mu^{(2)}(H) \geq \Omega(\sqrt{k(n-k)/n})$, and they noted that  similar conclusion holds for higher uniformities as well. Lenz and Mubayi \cite{LenzMubayi} also considered $\mu^{(2)}(H)$ in relation to hypergraph quasirandomness. While $\mu^{(2)}(H)$ might be a good measure of the second eigenvalue of dense hypergraphs, it becomes unusable for sparse ones. Indeed, if $H$ is the random 3-uniform hypergraph on $n$ vertices, in which each edge is included with probability $p$, then with high probability $\mu^{(2)}(H)=\Theta(\sqrt{pn})$ if $1/n\ll p \ll 1/2$, but $\mu^{(2)}(H)=\Theta(1)$ if $p\ll 1/n$.   

Another analogue of the Alon-Boppana theorem for $d$-regular hypergraphs is proposed by Li and Mohar \cite{LM}. They study the quantity $\mu^{(r)}(H)$, and prove that if $H$ is an $r$-uniform $d$-regular hypergraph, then 
$$\mu^{(r)}(H) \geq \frac{r}{r-1}\left((r-1)(d-1)\right)^{1/r}-o_n(1).$$
Here, the quantity $\frac{r}{r-1}\left((r-1)(d-1)\right)^{1/r}$  coincides with the spectral radius of the infinite $d$-regular hypertree by a result of Friedman \cite{hypertree}. Therefore, it is tempting to believe that this is the right quantity due to its attractive parallel with the graph case.  However, as the first part of our next theorem shows, even $\lambda_2^{(r)}(H)$ is always at least $\Omega_r(\sqrt{d})$ for any uniformity, thus greatly improving the result of Li and Mohar. We also note that the analogous bound can be achieved for $\mu^{(p)}(H)$ in a slightly wider range of $d$. This gives an alternative proof of the result of Friedman and Wigderson \cite{FW} for $3$-uniform hypergraphs, and also extends the result to regular hypergraphs.

\begin{theorem}\label{thm:2nd eigenvalue}
Let $H$ be a $d$-regular $r$-uniform hypergraph on $n$ vertices, $r\geq 3$.
    \begin{itemize}
        \item[(i)] 
          If $d \leq n^{2/3}$, then $$\lambda_2^{(r)}(H) \geq c\sqrt{d}$$ for some $c=c(r)>0$ depending only on $r$. In general, for any $p\geq 1$, $$\lambda_2^{(p)}(H)\geq cn^{1-r/p}\sqrt{d}.$$
        \item[(ii)]
         If $d <\frac{n^2}{4}$, then $$\mu^{(r)}(H) \geq c\sqrt{d}$$ for some $c = c(r) > 0$ depending only on $r$. In general, for any $p \geq 1$, 
        $$\mu^{(p)}(H) \geq c n^{1-r/p} \sqrt{d}.$$

        \item[(iii)]
        If $r\geq 4$ and $\frac{n^2}{4}\leq d\leq \frac{1}{2}\binom{n-1}{r-1}$, then 
         $$\mu^{(r)}(H) \geq \frac{cd}{n}$$
         for some $c = c(r) > 0$ depending only on $r$. In general, for any $p \geq 1$, 
        $$\mu^{(p)}(H) \geq cdn^{-r/p}$$
    \end{itemize}
\end{theorem}

This result follows almost immediately from our bounds on the discrepancy. Indeed, it turns out that $\mu^{(p)}(H)$ controls the discrepancy of $H$, while $\lambda_2^{(p)}(H)$ controls the positive discrepancy. In particular, we have the following relationship between these quantities.

\begin{lemma}
Let $H$ be an $r$-uniform hypergraph on $n$ vertices of average degree $d$. Then
$$n^{r/p}\cdot\lambda_2^{(p)}(H)\geq r\disc^{+}(H)-O_r(d)$$
and 
$$n^{r/p}\cdot \mu^{(p)}(H)\geq r\disc(H)-O_r(d).$$
\end{lemma}

\begin{proof}
 Let $U\subset V(H)$, and let $y$ be the characteristic vector of $U$, then $x=|U|^{-1/p}\cdot y$ satisfies $||x||_p=1$. Also, 
 \begin{equation}\label{equ:char vector}
    \sigma_H(x,\dots,x)=|U|^{-r/p}\cdot \left(r e(U)-\frac{dn}{n^r}\cdot |U|^r\right).   
 \end{equation}
Here, we have the following relationship between the discrepancy and normalized adjacency map:
$$r\disc(U)-|U|^{r/p}\sigma_H(x,\dots,x)=dn\cdot\left(\frac{|U|^{r-1}}{n^{r-1}}-\frac{|U|\dots (|U|-r+1)}{n\dots (n-r+1)}\right)=O_r(d).$$
Therefore, choosing $U$ satisfying $\disc(U)=\disc^{+}(H)$, or $\disc(U)=\disc(H)$, verifies the two desired inequalities.
\end{proof}

Hence, as long as $\disc^{+}(H)=\Omega(\sqrt{d}n)\gg d$, which is satisfied for $d\ll n^{2/3}$ by Theorem \ref{thm:main_pos_disc}, the inequality $\lambda_2^{(p)}(H)\geq cn^{1-r/p}\sqrt{d}$ follows from the previous lemma. Similarly, the second inequality follows by using Theorem \ref{thm:main_disc}, and observing that there exists $\varepsilon > 0$ for which $r\disc^{+}(H) - O_{r}(d) = \Omega_{r}(\sqrt{d}n)$ provided that $d < \varepsilon n^2$. 

When $\varepsilon n^2\leq d\leq \frac{1}{2}\binom{n-1}{r-1}$, one can show the improved lower bound $\mu^{(p)}(H) = \Omega_r\left(n^{-r/p}d\right)$. Indeed, let $U$ be a uniformly random subset of $V(G)$ of size $\lceil \frac{n}{2} \rceil$, let $y$ denote the characteristic vector of $U$, and let $x = |U|^{-1/p} \cdot y$. Then $||x||_p = 1$, and $x$ also satisfies equation (\ref{equ:char vector}). Since 
$$\mathbb{E} \left(e(U) - \frac{dn}{rn^r}\cdot \frac{|U|^r}{n^r}\right) = p\binom{|U|}{r} - p\binom{n}{r}\frac{|U|^r}{n^r} \leq -c p n^{r-1}$$
for some constant $c>0$ depending only on $r$, it follows that there exists a set $U$ for which we have 
$$|\sigma(x,\dots,x)| \geq \Omega(n^{-r/p}d),$$
which implies the desired bound $\mu^{(p)}(H) = \Omega_r(n^{-r/p}d)$. 


\section{Proof overview}

Let us give a short overview of our proofs, in particular the proofs of Theorem \ref{thm:bisection} and Theorem \ref{thm:main_pos_disc}. We follow the ideas of \cite{RST}, which in turn were inspired by the semidefinite programming approach of Goemans and Williamson \cite{GW95} on the MaxCut problem.

Let $H$ be an $r$-uniform $d$-regular hypergraph. We assign certain unit vectors in $\mathbb{R}^{V(H)}$ to the vertices of $H$ with the property that if the vertices  $v$ and $w$ are both contained in some edge, then the scalar product of their corresponding vectors is slightly positive, in particular at least $\Omega(d^{-1/2})$. Then, we chose a random linear half-space, and define $X\subset V(H)$ to be the set of vertices whose corresponding vectors are contained in this half-space. The vectors are constructed in a manner to ensure that $r$-tuples of vertices forming an edge are more likely to be contained in $X$ than the average $r$-tuple. In particular, we show that the expected discrepancy of $X$ is $\Omega_r(\sqrt{d}n)$, proving Theorem \ref{thm:main_pos_disc} for regular hypergraphs. To prove the general statement, we argue that large degree vertices can be either omitted, or already contribute large discrepancy. In order to prove Theorem \ref{thm:bisection}, we further argue that the size of $X$ must be close to $n/2$. Hence, we can add or remove a few vertices  to get a set $X'$ of size $\lfloor n/2\rfloor $. We show that the bisection given by the partition $X' \cup (V(H)\setminus X')$ is of size $\frac{nd}{r}\left(1 - \frac{1}{2^{r-1}}\right) - \Omega(\sqrt{d}n)$ in expectation.

In order to execute this strategy, we first need a bound on the probability that given $r$ vectors $v_1,\dots,v_r$, they are simultaneously contained in a random linear half-space. This problem is discussed in the next subsection.

\section{A probabilistic geometric lemma}

In this section, we consider the following probabilistic problem in geometry, which is the backbone of our proofs. 

\begin{problem*}
Given $r$ vectors $v_1,\dots,v_r\in \mathbb{R}^n$, what is the probability that they are simultaneously contained in a random linear half-space?
\end{problem*}

\noindent
 To this end, let $\mathbf{w}$ be a random unit vector in $\mathbb{R}^n$, chosen  from the uniform distribution. Define
$$\mu(v_1,\dots,v_r):=\mathbb{P}\left( \langle \mathbf{w},v_i\rangle\geq 0\mbox{ for every } i\in [r]\right).$$
In case $r=2$, this probability is easy to calculate: $\mu(v_1,v_2)=\frac{\pi-\alpha}{2\pi}$, where $\alpha$ is the angle between $v_1$ and $v_2$. However, for $r\geq 3$, we are unable to provide an easy to use formula. Note that when $v_1,\dots,v_r$ are pairwise orthogonal, then $\mu(v_1,\dots,v_r)=1/2^r$. In the next lemma, we show that under some mild assumptions,
$$\mu(v_1,\dots,v_r)=\frac{1}{2^r}+\Theta_r\left(\sum_{1\leq i <j\leq r}\langle v_i,v_j\rangle\right).$$

\begin{lemma}\label{lemma:geometry}
For every $r$, there exist $0<c_1<c_2$ and $\alpha>0$ such that the following holds. Let $v_1,\dots,v_r\in \mathbb{R}^n$ be unit vectors such that $0\leq \langle v_i,v_j\rangle\leq \alpha$ for every $i,j\in [r]$. Then
$$\mu(v_1,\dots,v_r)\in\frac{1}{2^r}+ [c_1,c_2]\cdot \sum_{1\leq i<j\leq r} \langle v_i,v_j\rangle .$$
\end{lemma}

\begin{proof}
Without loss of generality, we may assume that $a=\langle v_1,v_2\rangle$ is maximal among $\langle v_i,v_j\rangle$, $1\leq i<j\leq r$. Then, our goal is to show that $\mu(v_1,\dots,v_r)=\frac{1}{2^r}+\Theta_r(a)$ assuming $a\leq \alpha$ is sufficiently small with respect to $r$.

 Let $H$ be an $r$-dimensional linear hyperplane containing $v_1,\dots,v_r$. Note that if $\mathbf{w'}$ is the projection of $\mathbf{w}$ onto $H$, then $\langle v_i,\mathbf{w}\rangle=\langle v_i,\mathbf{w'}\rangle$ and $\frac{\mathbf{w'}}{||\mathbf{w'}||_2}$ is uniformly distributed on the unit sphere of $\mathbb{R}^r$. Hence, we may assume that $n=r$. Furthermore, note that if $A$ is an isometry of $\mathbb{R}^r$, then $\langle Ax,Ay\rangle=\langle x,y\rangle$ for any $x,y\in \mathbb{R}^r$, and $A\mathbf{w}$ has the same distribution as $\mathbf{w}$. Hence, after applying a suitable isometry to the vectors $v_1,\dots,v_r$, we may assume that the matrix with rows $v_1,\dots,v_r$ is lower-triangular with non-negative diagonal entries. In other words, $v_i(i)\geq 0$ and $v_i(i+1)=v_i(i+2)=\dots=v_i(r)=0$ for $i\in [r]$. Note that the numbers $\langle v_i,v_j\rangle$ then uniquely determine the $r$ vectors $v_1,\dots,v_r$. Finally, we may assume that $\mathbf{w}$ is chosen randomly in the unit ball of $\mathbb{R}^r$ from the uniform distribution, instead of the unit sphere. 

Recall that $a=\langle v_1,v_2\rangle$. Then $v_1=(1,0,\dots,0)$ and $v_2=(a,\sqrt{1-a^2},0,\dots,0)$. Also, by the maximality of $a$, $$a\leq \sum_{1\leq i<j\leq r} \langle v_i,v_j\rangle\leq r^2 a.$$
Next, let us bound the entries of $v_{\ell}$ for $\ell\in [r]$.

\begin{claim}\label{claim:prob}
Let $\ell \in [r]$. Then $v_{\ell}(\ell)\geq 1/2$. Also, there exists $c=c(r)>0$ such that if $i< \ell$, then $v_{\ell}(i)\in [-ca^2,ca]$.
\end{claim}
\begin{proof}
Assume that $\alpha<\frac{1}{18r}$. We prove the following statement by double induction, first on $\ell$, then on $i$: for every  $\ell\geq 2$, if $i\in \{1,\dots,\ell-1\}$, then $v_{\ell}(i)\in [-18ra^2,3a]$, and $v_{\ell}(\ell)\in [1/2,1]$. As $$v_{\ell}(\ell)=\left(1-\sum_{i=1}^{\ell-1}v_{\ell}(i)^2\right)^{1/2}\geq \left(1-9ra^2\right)^{1/2},$$ the inequality $v_{\ell}(\ell)\geq 1/2$ follows by noting that $a\leq \alpha$ and assuming that our induction hypothesis holds for $i\leq \ell-1$.

 Note that $v_{\ell}(1)=\langle v_1,v_{\ell}\rangle\in [0,a]$, so the statement is true for $\ell\geq 2$ and $i=1$. Now fix $\ell\geq 3$ and $2\leq i<\ell$, and assume that our induction hypothesis is true for every pair $(\ell_0,i_0)$ with $\ell_0<\ell$ or $\ell_0=\ell$ and $i_0<i$. Noting that $v_i(j)=0$ if $j>i$, we have 
$$\langle v_i,v_{\ell}\rangle =\sum_{k=1}^{i} v_i(k)v_{\ell}(k).$$ 
Hence,
$$v_i(i) v_{\ell}(i)=\langle v_i,v_{\ell}\rangle-\sum_{k=1}^{i-1} v_i(k)v_{\ell}(k).$$
Here, $v_{i}(i)\in [1/2,1]$, $\langle v_i,v_{\ell}\rangle\in [0,a]$, and each term $v_i(k)v_{\ell}(k)$ in the sum is in $[-9a^2,9a^2]$. Hence, $v_i(i)v_{\ell}(i)\in [-9ra^2,a+9r a^2]$, and $v_{\ell}(i)\in [-18ra^2,2a+18ra^2]$. The statement follows as $a\leq 1/(18r)$.
\end{proof}

Let $B_1$ denote the unit ball in $\mathbb{R}^r$, let $B_1^+\subset B_1$ be the set of vectors with nonnegative coordinates, and let $S\subset B_1$ be the set of vectors $w$ for which $\langle v_i,w\rangle\geq 0$ holds for every $i\in [r]$. Note that $\vol(B_1^+)/\vol(B_1)=1/2^{r}$ and $\mu(v_1,\dots,v_r)=\vol(S)/\vol(B_1)$, where $\vol(.)$ denotes the volume.

First, we establish an upper bound on $\mu(v_1,\dots,v_r)$. Let $c>0$ be the constant given by the previous claim, and for $i=1,\dots,r$, let $$R_i=\{z\in [-1,1]^r: -2cra\leq z(i)\leq 0\}.$$ We claim that $S\subset B_1^+\cup R_1\cup\dots \cup R_r$. Indeed, let $w\in B_1$ be a vector, and assume that $w(\ell)<-2cra$ for some $\ell\in [r]$. Then 
$$\langle w,v_{\ell}\rangle=\sum_{i=1}^{\ell}w(i)v_{\ell}(i)\leq (\ell-1)ca+\frac{w(\ell)}{2}<0$$
by Claim \ref{claim:prob}, so $w\not\in S$. But $$\vol(B_1^+\cup R_1\cup\dots\cup R_r)\leq \frac{\vol(B_1)}{2^r}+r\cdot(2^{r-1}\cdot (2cra)).$$
Therefore,
$$\mu(v_1,\dots,v_r)=\frac{\vol(S)}{\vol(B_1)}\leq \frac{1}{2^r}+\frac{2^r r^2 ca}{\vol(B_1)}\leq \frac{1}{2^r}+c_2\sum_{1\leq i<j\leq r} \langle v_i,v_j\rangle$$
 with $c_2=\frac{2^r r^2 c}{\vol(B_1)}$.

Next, we establish the lower bound on $\mu(v_1,\dots,v_r)$. First, let $Q$ be the set of vectors in $B$ whose every coordinate is at least $2rca^2$. If $w\in Q$, then 
 $$\langle v_{\ell},w\rangle=\sum_{i=1}^{\ell}v_{\ell}(i)w(i)\geq \left(\sum_{i=1}^{\ell-1}-ca^2\right)+w(\ell)v_{\ell}(\ell)\geq 0$$
 by Claim \ref{claim:prob}, so $Q\subset S$. Furthermore, $\vol(Q)\geq \vol(B_1^+)-2r^2ca^2$. Let 
 $$R=\left[\frac{1}{2r},\frac{1}{r}\right]\times \left[-\frac{a}{2r},0\right]\times \left[\frac{1}{2r},\frac{1}{r}\right]^{r-2}.$$ Observe that $R\subset B_1$. For every  $w\in R$, we have  
 $$\langle v_2,w\rangle=aw(1)+\sqrt{1-a^2}\cdot w(2)\geq \frac{a}{2r}-\frac{a}{2r}=0,$$
 and if $\ell \neq 2$, then
 $$\langle v_{\ell},w\rangle=v_{\ell}(\ell)w(\ell)+v_{\ell}(2)w(2)+\sum_{i\leq \ell-1, i\neq 2}v_{\ell}(i)w(i)\geq \frac{1}{4r}-\frac{ca^2}{2r}-(\ell-2)\cdot \frac{ca^2}{r}\geq 0.$$
 Thus $R\subset S$ holds as well. But then $Q\cup R\subset S$, and we have 
 \begin{align*}
     \mu(v_1,\dots,v_r)&=\frac{\vol(S)}{\vol(B_1)}\geq \frac{\vol(Q)}{\vol(B_1)}+\frac{\vol(R)}{\vol(B_1)}\\
     &\geq \frac{1}{2^r}-\frac{2rca^2}{\vol(B_1)}+\frac{a}{(2r)^r\vol(B_1)}\geq \frac{1}{2^r}+c_1\sum_{1\leq i<j\leq r} \langle v_i,v_j\rangle
 \end{align*}
 with suitable $c_1>0$, assuming $a\leq \alpha$ is sufficiently small.
\end{proof}

\section{Bisection width}


In this section, we prove Theorems \ref{thm:bisection} and \ref{thm:mixed}. We prepare the proof with a technical lemma, which is also the key result in the proof of Theorem \ref{thm:main_pos_disc}. But first, let us recall the definition of discrepancy and positive discrepancy.

\begin{definition}
Let $H$ be an $r$-uniform hypergraph with $n$ vertices and edge density $p=\frac{|E(H)|}{\binom{n}{r}}$. Given $U\subset V(H)$, define the discrepancy of $U$ as 
$$\disc(U)=e(U)-p\binom{|U|}{r}.$$
Then the \emph{positive discrepancy} of $H$ is defined as
$$\disc^{+}(H)=\max_{U\subset V(H)}\disc(U).$$
Similarly, the \emph{negative discrepancy} of $H$ is
$$\disc^{-}(H)=\max_{U\subset V(H)}-\disc(U),$$
and the \emph{discrepancy} is $$\disc(H)=\max_{U\subset V(H)} |\disc(U)|=\max\{\disc^{+}(H),\disc^{-}(H)\}.$$
\end{definition}

Now we are ready to state our key lemma.

\begin{lemma}\label{lemma:disc_maxdeg}
Let $H$ be an $r$-uniform hypergraph on $n$ vertices of maximum degree $\Delta$. 
\begin{itemize}
    \item[(i)] If $\Delta\leq n^{2/3}$, then for some $c=c(r)>0$, $$\disc^{+}(H)\geq \frac{ce(H)}{\sqrt{\Delta}}.$$

    \item[(ii)] If $e(H)>C\Delta^2 \sqrt{n}$ for some sufficiently large $C=C(r)>1$, then the vertex set of $H$ can be partitioned into two parts, $X$ and $Y$, such that for some $c'=c'(r)>0$,
$$e(X)+e(Y)-\Delta\cdot||X|-|Y||\geq e(H)\cdot\left(\frac{1}{2^{r-1}}+\frac{c'}{\sqrt{\Delta}}\right).$$
\end{itemize}

\end{lemma}

\begin{proof}
Let $\alpha=\alpha(r)$ satisfying $0<\alpha<\min\{0.1,\alpha_0\}$ be specified later, where $\alpha_0$ is the constant given by Lemma \ref{lemma:geometry} as $\alpha$. We may assume that $\Delta$ and $n$ are sufficiently large with respect to $\alpha$, so in turn, with respect to $r$. Also, let $p=e(H)/\binom{n}{r}$ be the edge density of $H$. Let $V=V(H)$, and for every vertex $v\in V$, assign the vector $x_v\in \mathbb{R}^V$ as follows: for $u\in V$,
$$x_v(u)=\begin{cases}1 &\mbox{ if }v=u,\\
\frac{\alpha}{\sqrt{2r\Delta}} &\mbox{ if there exists }e\in E(H)\mbox{ such that } v,u\in e,\\
0&\mbox{ otherwise.}\end{cases}$$
Note that 
$$1\leq ||x_v||_2^2\leq 1+(r-1)\cdot \Delta\cdot \left(\frac{\alpha}{\sqrt{2r\Delta}}\right)^2\leq 2.$$
Let $y_v=x_v/||x_v||_2$ be the normalization of $x_v$. Clearly, $0\leq \langle y_u,y_v\rangle$ for any $u,v\in V$, and if $u\neq v$, then 
$$\langle y_u,y_v\rangle\leq \langle x_u,x_v\rangle\leq  \frac{2\alpha}{\sqrt{2r\Delta}}+(r-1)\cdot\Delta\cdot \left(\frac{\alpha}{\sqrt{2r\Delta}}\right)^2<\alpha.$$
Furthermore, if $u$ and $v$ both appear in some edge $e$, then $$\langle y_u,y_v\rangle\geq \frac{1}{2}\langle x_u,x_v\rangle\geq \frac{\alpha}{\sqrt{2r\Delta}}.$$ Let $\mathbf{w}$ be a random unit vector in $\mathbb{R}^V$, chosen from the uniform distribution, and define $$X=\{v\in V:\langle y_v,\mathbf{w}\rangle\geq 0\}.$$

First, we calculate the expected discrepancy of $X$ in $H$. We have
$$\mathbb{E}(\disc(X))=\mathbb{E}(e(X))-p\cdot \mathbb{E}\binom{|X|}{r}.$$
For each $r$-element set $e=\{v_1,\dots,v_r\}\subset V$, we have $\mathbb{P}(e\subset X)=\mu(y_{v_1},\dots,y_{v_r})$. As the vectors $y_{v_1},\dots,y_{v_r}$ satisfy the required conditions of Lemma~\ref{lemma:geometry}, we can write that 
$$\frac{1}{2^r}+c_1 \sum_{1\leq i<j\leq r} \langle y_{v_i},y_{v_j}\rangle\leq \mathbb{P}(e\subset X)\leq \frac{1}{2^r}+c_2 \sum_{1\leq i<j\leq r} \langle y_{v_i},y_{v_j}\rangle,$$
where $c_1,c_2>0$ are suitable constants depending only on $r$. Therefore, if $e$ is an edge of $H$, then $$\mathbb{P}(e\subset X)\geq \frac{1}{2^r}+\frac{c_1\alpha}{\sqrt{2r\Delta}},$$
and so
\begin{equation}\label{equ:expected_edges}
    \mathbb{E}(e(X))\geq e(H)\left(\frac{1}{2^r}+\frac{c_1\alpha}{\sqrt{2r\Delta}}\right).
\end{equation}
On the other hand, $\mathbb{E}\binom{|X|}{r}$ is equal to the expected number of $r$-element sets contained in $X$, so 
$$\mathbb{E}(\disc(X))\geq e(H)\cdot \left(\frac{1}{2^r}+\frac{c_1\alpha}{\sqrt{2r\Delta}}\right)-p\sum_{\{v_1,\dots,v_r\}\in V^{(r)}}\left(\frac{1}{2^r}+c_2\sum_{1\leq i<j\leq r}\langle y_{v_i},y_{v_j}\rangle\right).$$
The terms containing $\frac{1}{2^r}$ cancel, so we get
$$\mathbb{E}(\disc(X))\geq\frac{c_1\alpha e(H)}{\sqrt{2r\Delta}}-c_2p\sum_{\{v_1,\dots,v_r\}\in V^{(r)}}\sum_{1\leq i<j\leq r}\langle y_{v_i},y_{v_j}\rangle.$$
Here, one can write
$$\sum_{\{v_1,\dots,v_r\}\in V^{(r)}}\sum_{1\leq i<j\leq r}\langle y_{v_i},y_{v_j}\rangle=\binom{n-2}{r-2}\sum_{\{v,v'\}\in V^{(2)}}\langle y_{v},y_{v'}\rangle=\binom{n-2}{r-2}\sum_{u\in V} \sum_{\{v,v'\}\in V^{(2)}}y_{v}(u)y_{v'}(u).$$
Now fix some $u\in V$, and let us bound  $\sum_{\{v,v'\}\in V^{(2)}}y_{v}(u)y_{v'}(u)$, which in particular is bounded by $\sum_{\{v,v'\}\in V^{(2)}}x_{v}(u)x_{v'}(u)$. There are at most $(r\Delta)^2$ pairs $\{v,v'\}$ such that both $\{u,v\}$ and $\{u,v'\}$ are contained in some edge of $H$. If $v\neq u$ and $v'\neq u$, each such pair contributes $\frac{\alpha^2}{2r\Delta}$ to the second sum. Also, there are at most $2r\Delta$ pairs $\{v,v'\}$ such that $v=u$ or $v'=u$, and both $\{u,v\}$ and $\{u,v'\}$ are contained in some edge. Each such pair contributes $\frac{\alpha}{\sqrt{2r\Delta}}$ to the sum. Hence,
\begin{equation*}
    \sum_{\{v,v'\}\in V^{(2)}}y_{v}(u)y_{v'}(u)\leq (r\Delta)^2\cdot \frac{\alpha^2}{2r\Delta}+2r\Delta\cdot \frac{\alpha}{\sqrt{2r\Delta}}\leq 4r\Delta\alpha^2,
\end{equation*}
where the last inequality holds by our assumption that $\Delta$ is sufficiently large with respect to $\alpha$. From this, we get
\begin{equation}\label{equ:scalar_sum}
\sum_{\{v,v'\}\in V^{(2)}}\langle y_{v},y_{v'}\rangle\leq 4r\Delta\alpha^2 n,
\end{equation}
and we can write
$$\mathbb{E}(\disc(X))\geq \frac{c_1\alpha e(H)}{\sqrt{2r\Delta}}-c_2p\binom{n-2}{r-2}\cdot  (4r\Delta\alpha^2 n)\geq \frac{c_3\alpha e(H)}{\sqrt{\Delta}}-\frac{c_4\alpha^2 e(H)\Delta}{n}$$
with suitable $c_3,c_4>0$ only depending on $r$. If $\Delta< n^{2/3}$, there is a choice for $\alpha$ depending only on $r$ such that the right-hand-side is at least $\frac{c_3\alpha}{2}\cdot \frac{e(H)}{\sqrt{\Delta}}$. Hence, we have $\mathbb{E}(\disc(X))\geq c_5\frac{e(H)}{\sqrt{\Delta}}$ with some constant $c_5>0$ depending only on $r$. Therefore, there is a choice for the random unit vector $\mathbf{w}$ such that the resulting set satisfies the required conditions of (i).

Now let us turn to the proof of (ii). Let $Y=V\setminus X$, then by symmetry, i.e.\ by noting that $X(\mathbf{w})=Y(-\mathbf{w})$ with probability 1, we have $\mathbb{E}(e(X))=\mathbb{E}(e(Y))$. Recall that by (\ref{equ:expected_edges}), we have $\mathbb{E}(e(X))=\mathbb{E}(e(Y))\geq e(H)(\frac{1}{2^r}+\frac{c_1\alpha}{\sqrt{\Delta}})$. It remains to bound the expectation of $||X|-|Y||=|2|X|-n|$. By convexity,
$$(\mathbb{E}(|2|X|-n|))^2\leq \mathbb{E}((2|X|-n)^2)=\mathbb{E}(n^2-4|X|n+4|X|^2).$$
For every vertex $v$, $\mathbb{P}(v\in X)=\frac{1}{2}$, so $\mathbb{E}(|X|)=n/2$. Furthermore, for every pair of distinct vertices $\{v,v'\}$, we have $\mathbb{P}(v,v'\in X)=\mu(y_v,y_{v'})\leq \frac{1}{4}+c_6 \langle y_v,y_{v'}\rangle$ for some constant $c_6>0$ by Lemma \ref{lemma:geometry} applied with $r=2$. Hence,
$$\mathbb{E}(|X|^2)=\sum_{v,v'\in V}\mathbb{P}(v,v'\in X)\leq \frac{n}{2}+2\sum_{\{v,v'\}\in V}\left[\frac{1}{4}+c_6\langle y_v,y_{v'}\rangle\right]\leq \frac{n^2}{4}+\frac{n}{4}+8c_6 r\Delta \alpha^2n,$$
where the last inequality follows by (\ref{equ:scalar_sum}). In conclusion,
$$\mathbb{E}((2|X|-n)^2)\leq n+32c_6r\Delta \alpha^2n,$$
and thus $\mathbb{E}(||X|-|Y||)\leq c_7 \sqrt{\Delta n}$ with some $c_7=c_7(r)>0$. 

Putting everything together,
$$\mathbb{E}(e(X)+e(Y)-\Delta||X|-|Y||)\geq \frac{e(H)}{2^{r-1}}+ \frac{2c_1\alpha e(H)}{\sqrt{\Delta}}-c_7 \Delta^{3/2}\sqrt{n}.$$
Assuming $C>0$ is sufficiently large, the condition $e(H)>C\Delta^2\sqrt{n}$ ensures that the right hand side is at least $e(H)\cdot (\frac{1}{2^{r-1}}+\frac{c_1\alpha}{\sqrt{\Delta}})$. But then there is a choice for the random unit vector $w$ such that the partition $X\cup Y$ satisfies the requirements of (ii).
\end{proof}

\begin{proof}[Proof of Theorem \ref{thm:bisection}]
Let $C,c'$ be the constants guaranteed by Lemma \ref{lemma:disc_maxdeg}, (ii). Setting $\varepsilon=\frac{1}{rC}$, the condition $d\leq \varepsilon n^{1/2}$ implies $\frac{dn}{r}=e(H)>Cd^2\sqrt{n}$. Then there exists a partition $X\cup Y$ of $V(H)$ such that 
$$e(X)+e(Y)-d||X|-|Y||\geq e(H)\left(\frac{1}{2^r}+\frac{c'}{\sqrt{d}}\right).$$Without loss of generality, we may assume that $|X|\leq |Y|$. Let $S$ be an arbitrary $\lfloor n/2\rfloor -|X|$ element subset of $Y$, and let $X_0=X\cup S$ and $Y_0=Y\setminus S$. Then $X_0\cup Y_0$ is a bisection of $H$, and 
$$e(X_0)+e(Y_0)\geq e(X)+e(Y)-|S|d= e(X)+e(Y)-d\left\lfloor\frac{|Y|-|X|}{2}\right\rfloor\geq e(H)\left(\frac{1}{2^r}+\frac{c'}{\sqrt{d}}\right).$$
Finally,
$$e(X,Y)=e(H)-e(X)-e(Y)\leq e(H)\left(1-\frac{1}{2^r}-\frac{c'}{\sqrt{d}}\right),$$
finishing the proof.
\end{proof}

 Straightforward  modifications of Lemma \ref{lemma:disc_maxdeg} imply Theorem \ref{thm:mixed} as well. We omit the details.
 
\section{Positive discrepancy}

In this section, we prove Theorem \ref{thm:main_pos_disc}, which we restate here for the reader's convenience.  
\begin{theorem}\label{thm:pos_disc}
Let $H$ be an $r$-uniform hypergraph on $n$ vertices of average degree $d<n^{2/3}$. Then
$$\disc^{+}(H)=\Omega_r(d^{1/2}n).$$
\end{theorem}

Note that in case the the maximum degree of $H$ is not much larger than its average degree, Lemma \ref{lemma:disc_maxdeg} immediately implies Theorem \ref{thm:pos_disc}. We show that the general case can be reduced to this special subcase. It will be useful to define the discrepancy of collections of sets as well. Given $s_1,\dots,s_k$ such that $s_1+\dots+s_k=r$ and disjoint sets $U_1,\dots,U_k$, define
$$e_{s_1,\dots,s_k}(U_1,\dots,U_k)=\# \{ e\in E(H):|e\cap U_i|=s_i\text{ for }i\in [k]\}$$
and
$$\disc_{s_1,\dots,s_k}(U_1,\dots,U_k)=e_{s_1,\dots,s_k}(U_1,\dots,U_k)-p\binom{|U_1|}{s_1}\dots\binom{|U_k|}{s_k}.$$
Note that if $U$ and $U'$ are disjoint, then
$$\disc(U\cup U')=\sum_{i=0}^{r}\disc_{i,r-i}(U,U').$$
Finally, let $\partial(X)$ be the set of all edges of $H$ that have a vertex in $X$. Then $$|\partial(X)|=\sum_{i=0}^{r-1}e_{r-i,i}(X,X^{c}),$$ where $X^c=V(H)\setminus X$ is the complement of $X$.  
Next, we prove a lemma which is used to handle large degree vertices of $H$.

\begin{lemma}\label{lemma:maxdeg}
		For every $r\geq 2$, there exist $c_1,C>0$ such that the following holds. Let $H$ be an $r$-uniform hypergraph on $n$ vertices of average degree $d$. Let $X\subset V(H)$ such that the degree of every vertex in $X$ is at least $Cd$. Then 
		$$\disc^{+}(H)\geq c_1 |\partial(X)|.$$
\end{lemma}
	
	\begin{proof}
		Let $p=\frac{d}{\binom{n}{r-1}}=\Theta(\frac{d}{n^{r-1}})$ be the density of $H$. Let $s=|X^c|$, and let $\alpha\in (0,1)$ be specified later, depending only on $\delta$ and $r$. Furthermore, let $b=\lfloor s/2\rfloor$, and let $\mathbf{Y}$ be a random $b$ element subset of $X^c$, chosen from the uniform distribution. We have
		$$\disc(X\cup \mathbf{Y})=\sum_{i=0}^{r}\disc_{r-i,i}(X,\mathbf{Y}).$$
		Here, 
        $$\mathbb{E}(\disc_{r-i,i}(X,\mathbf{Y}))=\left(\prod_{j=0}^{i-1}\frac{b-j}{s-j}\right)\cdot\disc_{r-i,i}(X,X^c).$$ Writing $\beta_i$ for the coefficient of $\disc_{r-i,i}(X,X^c)$ in the previous line, we thus get
		\begin{align*}
			\mathbb{E}(\disc(X\cup \mathbf{Y}))&=\sum_{i=0}^{r}\beta_i\disc_{r-i,i}(X,X^c)=\sum_{i=0}^{r-1}(\beta_i-\beta_r)\disc_{r-i,i}(X,X^c),
		\end{align*}
		where we used $0=\disc(X\cup X^c)=\sum_{i=0}^{r}\disc_{r-i,i}(X,X^c)$ in the second equality. Here, the right hand side can be  written as
		\begin{equation}
		\label{equ:cc}
				\sum_{i=0}^{r-1}(\beta_i-\beta_r)e(X,X^c)-p\left(\sum_{i=0}^{r-1}(\beta_i-\beta_r)\binom{|X|}{r-i}\binom{|X^c|}{i}\right)
		\end{equation}
		Next, we use that $\beta_i-\beta_r\geq (r-i)(\beta_{r-1}-\beta_r)$, which easily follows from the fact that $\beta_i\geq 2\beta_{i+1}$ for every $i=0,\dots,r-1$. Also, $\binom{|X|}{r-i}\binom{|X^c|}{i}\leq |X| n^{r-1}$ for $i=0,\dots,r-1$. Hence, we can lower bound (\ref{equ:cc}) as
   $$(\beta_{r-1}-\beta_r)\left(\sum_{i=0}^{r-1}(r-i)e_{r-i,i}(X,X^c)\right)-cd|X|$$
   with a suitable $c=c(r)>0$. But observe that $\sum_{i=0}^{r-1}(r-i)e_{r-i,i}(X,X^c)$ is just the sum of degrees of the vertices of $X$, which is at least $Cd|X|$ by our assumption on $X$. Hence, choosing $C>2c$, we conclude that
   $$\mathbb{E}(\disc(X\cup \mathbf{Y}))\geq \frac{\beta_{r-1}-\beta_r}{2}\left(\sum_{i=0}^{r-1}(r-i)e_{r-i,i}(X,X^c)\right)\geq \frac{\beta_{r-1}-\beta_r}{2} |\partial(X)|,$$
   so $c_1= \frac{\beta_{r-1}-\beta_r}{2}$ suffices.
	\end{proof}

\begin{proof}[Proof of Theorem \ref{thm:pos_disc}]
Let $\Delta=Cd$, where $C$ is given by Lemma \ref{lemma:maxdeg}. Let $X$ be the set of vertices of $H$ of degree more than $\Delta$, and let $H'$ be the hypergraph we get by removing every edge of $H$ having a vertex in $X$. Note that the maximum degree of $H'$ is at most $\Delta$, and $e(H')=e(H)-|\partial(X)|$. Hence, in case $e(H')\leq e(H)/2$, we have $|\partial(X)|\geq e(H)/2$, which gives $\disc^{+}(H)\geq \Omega_r(e(H))$ by Lemma \ref{lemma:maxdeg}, and we are done. Hence, we may assume that $e(H')\geq e(H)/2$.

By Lemma \ref{lemma:disc_maxdeg}, $$\disc^+(H')\geq \frac{c_2 e(H')}{\sqrt{\Delta}}\geq c_3 \sqrt{d} n$$ with some $c_2,c_3>0$ depending only on $r$. 

Now let us bound the discrepancy of $H$. If $|\partial(X)|\geq  \frac{1}{2}\disc^{+}(H')$, then  Lemma \ref{lemma:maxdeg} states that $$\disc^+(H)\geq c_1 |\partial(X)|\geq \frac{c_1 c_3}{2}\sqrt{d}n,$$ so we are done. Hence, we may assume that $|\partial(X)|\leq \frac{1}{2}\disc^{+}(H')$. Let $U\subset V(H)$ be such that $\disc^{+}(H')=\disc_{H'}(U)$, and let $p'=e(H')/\binom{n}{r}$ be the density of $H'$. Then
$$\disc_H(U)-\disc_{H'}(U)=e_H(U)-e_{H'}(U)-(p-p')\binom{|U|}{r}\geq -(p-p')\binom{n}{r}=-|\partial(X)|,$$
hence
$$\disc_H(U)\geq \disc^+(H')-|\partial(X)|\geq \frac{1}{2}\disc^+(H')\geq \frac{c_3}{2}\sqrt{d}n,$$
so we are done in this case as well.
\end{proof}

\section{Discrepancy}


In this section, we prove Theorem \ref{thm:main_disc}. A key preliminary tool in our proof is to observe that the theorem of Bollob\'as and Scott \cite{BS06} concerning the product of the positive and negative discrepancy of hypergraphs also extends to multi-hypergraphs. We now state the version of the theorem that is needed for our purposes. For completeness, the full proof of this theorem is added to  the Appendix. 

\begin{theorem}[Multi-hypergraph version of Theorem 14, \cite{BS06}]
\label{thm:BollobasScott main thm}Let $H$ be an $r$-uniform multi-hypergraph
of order $n$ with $p{\binom{n}{r}}$ edges, counted with multiplicities.
Suppose that $p$ satisfies the condition $\frac{1}{2n} \leq p \leq 1 - \frac{1}{2n}$. Then  
$$\disc^{+}\left(H\right)\cdot\disc^{-}\left(H\right)\geq\Omega\left(p\left(1-p\right)n^{r+1}\right).$$
\end{theorem}

Given disjoint sets $X$ and $Y$ and an integer $i\in\{0,\dots,r\}$, recall that we write 
$$\disc_{i,r-i}\left(X,Y\right)=e_{i,r-i}(X,Y)-p\binom{|X|}{i}\binom{|Y|}{r-i}.$$
If the host hypergraph $H$ is not clear from the context, we highlight it by writing $\disc_{i,r-i}(H;X,Y)$ or $e_{i,r-i}(H;X,Y)$ instead.

\begin{lemma}[Multi-hypergraph version of Lemma 9, \cite{BS06}]
\label{lem:BollobasScott lemma}
Let $H$ be an $r$-uniform multi-hypergraph and let $D=\disc(H)$. Then for every pair of disjoint sets $X$ and $Y$ and $i\in [r]$, we have
$$|\disc_{i,r-i}\left(X,Y\right)|\leq r^{2r}D.$$ 
\end{lemma}

\begin{proof}
Let $p\in [0,1]$, and let $Z$ be a random subset of $X$, each element chosen independently with probability $p$. Then
$$f(p):=\mathbb{E}(\disc(Z\cup Y))=\sum_{i=0}^{r}p^i\disc_{i,r-i}(X,Y).$$
Thus,  $f$ is a degree $r$ polynomial, and $f(p)\in [-D,D]$ for every $p\in [0,1]$. This implies that every coefficient of $f$ is at most $2^r r^{2r}D/r!< r^{2r}D$ in absolute value, see \cite{BE}. This finishes the proof.
\end{proof}

\begin{proof}[Proof of Theorem \ref{thm:main_disc}]
By taking complements if necessary, we may assume that the density of $H$ is at most $\frac{1}{2}$. For each $t \in \left\{2,\dots,r\right\}$, define the $t$-uniform multi-hypergraph $H_t$ by setting the multiplicity of $f \in V^{(t)}$ to be 
$$m(f) = \left \vert \left\{e \in E(H): f \subseteq e\right\}\right \vert.$$
Note that $H_r$ is simply $H$ itself.

Let $d_t$ denote the average degree of $H_t$, $p_t$ denote the edge density, and let $d=d_r$ and $p=p_r$. Since 
$$ p_t \binom{n}{t} = e(H_t) = p\binom{n}{r}\binom{r}{t},$$
it follows that 
$$p_t = \binom{n-t}{r-t}p.$$
Similarly, it is easy to conclude that we have 
$$d_t = \binom{r-1}{t-1} d.$$
In particular, it is straightforward to verify that for all $t \in \left\{2,\dots,r\right\}$,  
\begin{equation}
\label{equ:ratios}
\frac{p_{t+1}}{p_t} = \frac{r-t}{n-t}.
\end{equation}
We start by proving that for each $H$, at least one of the graphs $H_t$ has a suitable density so that Theorem \ref{thm:BollobasScott main thm} applies. 
\begin{claim}
    There exists $t \in \left\{2,\dots,r\right\}$ for which we have $\frac{1}{2n} \leq p_t \leq \frac{1}{2}$.
\end{claim}
\begin{proof}
When $d = \frac{1}{2} \binom{n-1}{r-1}$, we can take $t = r$ as one clearly has $p_r = \frac{1}{2}$. When $d = 1$, it is easy to verify that $e(H_2) = \binom{r}{2} e(H) \geq \frac{n(r-1)}{2}$. Hence it follows that $p_2 \geq \frac{r-1}{n-1}$. But equation $\eqref{equ:ratios}$  implies that $p_2\geq \dots\geq p_r$ and $p_{t+1}>np_t$, so taking the smallest index $t$ satisfying $p_t\geq \frac{1}{2n}$, we have $p_t\leq 1/2$ as well.  
\end{proof}
Since $H_t$ satisfies the conditions of Theorem \ref{thm:BollobasScott main thm} and $d = \Theta_r(d_t)$, we have $\disc\left(H_t\right) = \Omega_{r}(\sqrt{d}n)$. Let $U \subseteq V(H)$ be chosen so that $\vert \disc_{H_t}\left(U\right) \vert = \disc\left(H_t\right)$. In order to infer results concerning $\disc(H)$, we rewrite the terms $e_{H_t}\left(U\right)$ and $p_t \binom{|U|}{t}$ occurring in the expression $\disc_{H_t}\left(U\right)$. First of all, observe that 
\begin{equation}
\label{equ:edges}
e_{H_t}(U) = \sum_{j=t}^{r} \binom{j}{t} e_{j,r-j}\left(H;U,U^c\right).
\end{equation}
By using standard identities for binomial coefficients, we also conclude that 
\begin{equation}
\label{equ:binomials}
p_t \binom{|U|}{t} = p \binom{n-t}{r-t} \binom{|U|}{t} = p \binom{j}{t} \sum_{j=t}^{r} \binom{|U|}{j} \binom{|U^c|}{r-j}.
\end{equation}
Combining equations \eqref{equ:edges} and \eqref{equ:binomials}, we conclude that 
$$|\disc_{H_t}(U)|=\left\vert \sum_{j=t}^{r} \binom{j}{t} \disc_{j,r-j}\left(H;U,U^c\right) \right\vert = \Omega_r(\sqrt{d}n).$$
Thus, by the triangle inequality there exists $j \in \left\{t,\dots,r\right\}$ for which
$$\left\vert\disc_{j,r-j}\left(H;U,U^c\right)\right\vert = \Omega_r(\sqrt{d}n).$$
But then Lemma \ref{lem:BollobasScott lemma} implies that $\disc(H) = \Omega_r(\sqrt{d}n),$ which completes the proof. 
\end{proof}



\section*{Appendix: Proof of Theorem \ref{thm:BollobasScott main thm}}

For completeness, we now present the proof of Theorem \ref{thm:BollobasScott main thm}. The proof follows the same lines as the proof in \cite{BS06}, with appropriate trivial modifications added to accommodate the fact we are dealing with multi-hypergraphs. 
In the proof, we also need four preliminary results from \cite{BS06}. The only one of these that is not identical to its counterpart in \cite{BS06} is the fourth one, whose proof we include.

Given an edge-weighting $w$ on a complete $r$-uniform hypergraph $H$ and disjoint sets $X_1,\dots,X_t \subseteq V(H)$, we define $d_{k_1,\dots,k_t}(X_1,\dots,X_t)$ by setting $d_{k_1,\dots,k_t}(X_1,\dots,X_t) = \sum_{e} w(e)$, where the sum is taken over all edges $e$ satisfying the condition $\vert e \cap X_i\vert  = k_i$  for each $1 \leq i \leq t$. In case no edge-weighting is given, we assume $w(e)=1$ for every edge. 

\begin{lemma}[Lemma 6, \cite{BS06}]
\label{lemma:Bernoulli}
Let $\varepsilon_i$ be i.i.d.\ Bernoulli random variables with $\varepsilon_i \in \{-1,1\}$, and let $\mathbf{a} = (a_i)_{i=1}^{n}$ be a sequence of real numbers. Then 
$$\mathbb{E}\left\vert \sum_{j=1}^n a_i \varepsilon_i \right\vert \geq \frac{\vert\vert \mathbf{a}\vert\vert_1}{\sqrt{2n}}.$$
\end{lemma}

\begin{lemma}[Lemma 10, \cite{BS06}]
    \label{lem: lemma 10}
    Let $H$ be a complete $r$-uniform hypergraph of order $n$ with edge-weighting $w$. Let $V(H) = U \cup W$ be a random bipartition, with each vertex assigned to one of the sides independently with probability $\frac{1}{2}$. Then 
    $$\mathbb{E} \sum_{K \in U_{r-1}^{(r-1)}} \vert d_{r-1,1}(K,W)\vert \geq r2^{-r} \sum_{L \in V(G)^{r}} |w(L)|/\sqrt{2n}.$$
\end{lemma}
\begin{lemma}[Lemma 11, \cite{BS06}]
    \label{lem:lemma 11}
    Let $H$ be a $r$-uniform hypergraph of order $n$ with edge-weighting $w$. Suppose that $\alpha \geq 1$ and that there exists disjoint sets $X, Y\subset V(H)$ with 
    $$d_{1,r-1}(X,Y) + \alpha d(Y) = M \geq 0.$$
    Then at least one of the following holds 
    \begin{itemize}
        \item[(i)] $\disc^{+}(H) = 2^{-3r^2}M/\alpha$, or
        \item[(ii)]  $\disc^{-}(H) = 2^{-3r^2}M\alpha$.
    \end{itemize}
\end{lemma}

\begin{lemma}[Multi-hypergraph version of Lemma 13, \cite{BS06}]
    \label{lemma:Lemma13}
    Let $H$ be an $r$-uniform multi-hypergraph of order $n$ with $p\binom{n}{r}$ edges with $\frac{1}{2n} \leq p \leq 1 - \frac{1}{2n}$ and $n$ sufficiently large. Let $V(H) = X \cup Y$ be a random bipartition. Then 
    $$\mathbb{E}_{K\in X^{(r-1)}} \left\vert d_{r-1,1}(K,Y) - p |Y|\right\vert = \Omega_r\left(\sqrt{p(1-p)} n^{r - 1/2}\right)$$
\end{lemma}
\begin{proof}
    As usual, we may assume that $p \leq \frac{1}{2}$. Given $K \in V(H)^{(r-1)}$,  we write $d(K)$ for the number of edges containing $K$, $s(K)$ for the number of vertices $v$ so that $K\cup\{v\}$ is an edge with multiplicity at least 1, and we define $r(K)=d(K) - p(n-r+1)$. We also write $m(e)$ for the multiplicity of an $r$-tuple $e \in V(H)^{(r)}$. Furthermore, for each $v \in V(H)$, we set $\rho_v \in \{0,1\}$ to be the indicator random variable of the event $v \in Y$, and we set $\varepsilon_v = 2\rho_v - 1 \in \{-1,1\}$. 
    
    Consider a fixed set $K \in V(H)^{(r-1)}$, and assume that $d(K) > 0$. As in \cite{BS06}, we deduce that
    \begin{equation*}
    \begin{aligned}
        \mathbb{E} \left\vert d_{r-1,1}(K,Y\setminus K) - p\vert Y \setminus K\vert \right\vert &= \mathbb{E} \left\vert \sum_{v\not \in K} \rho_v\left(m(K\cup\{v\}) - p\right)\right\vert  \\&= \mathbb{E} \left\vert \frac{1}{2} \sum_{v \not \in K} (m(K\cup\{v\}) - p) + \frac{1}{2}\sum_{v \not \in K} \varepsilon_v(m(K\cup\{v\})-p)\right\vert
        \\&\geq \frac{1}{2}\max\left(|r(K)|, \mathbb{E}\left\vert \sum_{v\not \in K} \varepsilon_v (m(K\cup\{v\})-p)\right\vert \right)
        \\& \geq \frac{1}{4}\left(\vert r(K) \vert + \mathbb{E} \left\vert \sum_{v \not \in K} \varepsilon_v(m(K\cup\{v\})-p)\right\vert \right)
    \end{aligned}
    \end{equation*}
    Label the vertices of $V(H)$ with $1,\dots,n$ so that $K = \{n-r+2,\dots,n\}$, and so that $m(K \cup \{i\})$ is positive if and only if $i \leq s(K)$. Thus, using that $\mathbb{E}|X+Y|\geq \mathbb{E}|X|$ for any random variable $X$ and independent Bernoulli random variable $Y$, we can write
    $$\mathbb{E}\left\vert \sum_{v\not \in K} \varepsilon_v (m(K\cup\{v\})-p)\right\vert \geq  \mathbb{E}\left\vert \sum_{i = 1}^{s(K)} \varepsilon_i (m(K\cup\{i\})-p)\right\vert.$$
    Since $m(K\cup\{i\}) - p > 0$ for every $i \leq s(K)$, it follows that 
    $$\sum_{i = 1}^{s(K)} \vert m(K\cup\{i\}) - p\vert  = d(K) - ps(K) \geq (1-p)d(K).$$
    Thus Lemma \ref{lemma:Bernoulli} implies that
    $$\mathbb{E}\left\vert \sum_{i = 1}^{s(K)} \varepsilon_i (m(K\cup\{i\})-p)\right\vert  \geq (1-p)\frac{d(K)}{\sqrt{2s(K)}}\geq (1-p)\sqrt{\frac{d(K)}{2}}.$$
    Note that this bound remains also true when $d(K) = 0$. 
    In particular, it follows that for every $K$ we have 
    $$\mathbb{E} \left\vert d_{r-1,1}(K,Y\setminus K) - p\vert Y \setminus K\vert \right\vert \geq \frac{1}{4}|r(K)| + (1-p)\frac{\sqrt{d(K)}}{8}.$$
    Hence we conclude that 
    \begin{equation*}
        \begin{aligned}
            \mathbb{E}_{K\in X^{(r-1)}} \left\vert d_{r-1,1}(K,Y) - p |Y|\right\vert &= \sum_{K \in V^{(r-1)}} \mathbb{P}\left(K \subseteq X\right) \cdot \mathbb{E} \left\vert  d_{r-1,1}(K,Y\setminus K) - p\vert Y \setminus K\vert \right\vert
            \\ &\geq 2^{-r-2}\sum_{K \in V^{(r-1)}} \left(\vert r(K) \vert + \frac{\sqrt{d(K)}}{2}\right)
        \end{aligned}
    \end{equation*}
    Our aim is to show that there exists a uniform constant $\alpha$ so that each term in the sum is bounded below by $\alpha \sqrt{pn}$ when $n$ is sufficiently large. Since $p \leq \frac{1}{2}$ and there are $\binom{n}{r-1}$ terms in the sum, this certainly implies the lemma. 

    Suppose that $n \geq 2r$, and first consider the case when $d(K) \geq 1$. Observe that the expression $\vert d(K) - p(n-r+1)\vert + \frac{\sqrt{d(K)}}{2}$ as a function of $d(K)$ is clearly increasing when $d(K) \geq p(n-r+1)$. It can also be easily shown to be decreasing for $d(K) \in \left[\frac{1}{16},p(n-r+1)\right]$. Thus whenever $d(K) \geq 1$, we conclude that 
    $$|r(K)| + \frac{\sqrt{d(K)}}{2} \geq \frac{\sqrt{p(n-r+1)}}{2} \geq \alpha_1 \sqrt{pn}$$ for some constant $\alpha_1>0$.

    On the other hand, when $d(K) = 0$, we have $|r(K)| = p(n-r+1)$. Since $\frac{1}{2n} \leq p \leq \frac{1}{2}$, it follows that $p(n-r+1) \geq \alpha_2 \sqrt{pn}$ for some constant $\alpha_2>0$. Thus we may take $\alpha = \min(\alpha_1,\alpha_2)$, and the result follows. 
\end{proof}
We are now ready to prove Theorem \ref{thm:BollobasScott main thm}.
\begin{proof}[Proof of Theorem \ref{thm:BollobasScott main thm}]
    The proof follows exactly the one presented in \cite{BS06}. To start with, let $F$ denote the complete $r$-uniform weighted hypergraph on $V(H)$ with edge-weight $w(e) = m(e) - p$, where $m(e)$ denotes the multiplicity of the edge $e \in H$. It is clear that we have $w(F) = 0$ and  $\disc^{\pm}(F) = \disc^{\pm}(H)$. We may also assume that $p \leq \frac{1}{2}$ and that $\disc^{-}(H) \geq \disc^{+}(H) = c_r\sqrt{p(1-p)}n^{(r+1)/2}/\alpha$, where $c_r$ is a constant to be chosen later and $\alpha\geq 1$.

    Define random sets $W_{r} = V(H) \supset W_{r-1} \supset \dots \supset W_1$ so that for each $i \leq r - 1$, $W_{i+1} = W_{i} \cup X_{i+1}$ is a random bipartition of $W_{i+1}$, where each vertex is assigned independently to either of the two vertex classes with probability $\frac{1}{2}$. Define weightings $w_{i}$ such that for every $K\in W_{i}^{(i)}$, 
    $$w_{i}(K) = d_{i,1,\dots,1}(K,X_{i+1},\dots,X_r),$$
    with the convention that $w_r = w$. Lemma \ref{lem: lemma 10} implies that $i$ we have 
    $$\mathbb{E} \sum_{K\in W_{i}^{(i)}} \vert w_{i}(K)\vert \geq (i+1)2^{-i-1} \sum_{L\in W_{i+1}^{(i+1}} |w_{i+1}(L)|/\sqrt{2n}.$$
    Note that Lemma \ref{lemma:Lemma13} implies that 
    $$\mathbb{E} \sum_{K\in W_{r-1}^{(r-1)}} |w_{r-1}(K)| = \Omega_r\left(\sqrt{p(1-p)}n^{r-1/2}\right).$$
    Combining these two observations, we conclude that 
    $$\mathbb{E} \sum_{x\in W_1} |w_{1}(x)| = \Omega_r\left(\sqrt{p(1-p)}n^{(r+1)/2}\right).$$
    Let $X_{1}^{+} = \left\{w \in W_1: d_{1,\dots,1}(w,X_2,\dots,X_r) > 0 \right\}$. Note that
    \begin{equation}
        \mathbb{E} d_{1,\dots,1}(X_1^{+},X_2,\dots,X_r) = \Omega_{r}\left(\sqrt{p(1-p)}n^{(r+1)/2}\right).
    \end{equation}
    Define $V_0$ by setting $V_0 = X_1^{+} \cup \bigcup_{i=2}^{r} X_i$- Given a non-empty set $S \subset \left\{2,\dots,r\right\}$,  let $V_S = \bigcup_{i \in S} X_i$ and 
    $$E_{S} = \left\{K \cup \left\{x\right\}: x \in X_1^{+}, K \in V_{S}^{r-1}, |K \cap X_i| > 0 \text{ } \forall i \in S \right\}.$$ 
    Note that the sets $E_{S}$ partition the edges in $V_0$ that intersect $X_1^{+}$ in exactly one vertex. We also write $d_S = \sum_{K\in E_S} w(K)$, and observe that $d_{1,r-1}(X_1^{+},V_S) = \sum_{\emptyset \neq T \subset S} d_T$ and $d_{\{2,\dots,r\}} = d_{1,\dots,1}(X_1^{+},X_2,\dots,X_r)$.

    Let $S_0$ be minimal with $\vert d_{S_0}\vert \geq (2k)^{-k+|S|}d_{\{2,\dots,r\}}$. As in $\cite{BS06}$, we conclude that 
    \begin{align*}
    \max_{S \subset \left\{2,\dots,r\right\}} \vert d_{1,r-1}(X_1^{+},V_S) \vert &\geq \vert d_{1,r-1}\left(X_1^{+},V_{S_0}\right) \vert
    \\ &\geq |d_{S_0}| - \sum_{\emptyset \neq T \subset S_0} |d_T|
    \\ &= \Omega_{r}(d_{\{2,\dots,r\}}).
    \end{align*}
    Hence we conclude that there exists $S \subset \left\{2,\dots,r\right\}$ with 
    $$\mathbb{E} \left \vert d_{1,r-1}\left( X_{1}^{+}, V_S \right) \right \vert = \Omega_{r}\left(\sqrt{p(1-p)}n^{(r+1)/2}\right).$$
    Define $X_{S}^{+} = \left\{x \in W_1: d_{1,r-1}(\{x\},V_S) > 0 \right\}$. Since $\mathbb{E} d_{1,r-1}(W_1,V_S) = 0$ and $\mathbb{E} d(V_S) = 0$, we conclude that 
    $$\mathbb{E} d_{1,r-1}(X_S^{+},V_S) + \alpha d(V_S) = \beta_r \sqrt{p(1-p)} n^{(r+1)/2},$$
    where $\beta_r$ is some constant depending only on $r$. Set $c_r = \beta_r \cdot 2^{-3r^2-1}$, and recall that the constant $\alpha$ is chosen so that 
    $$\disc^{+}(H) = c_r \sqrt{p(1-p)}n^{(r+1)/2}/\alpha.$$
    If $\alpha \leq 1$, we are certainly done, as in this case both $\disc^{+}(H)$ and $\disc^{-}(H)$ are $\Omega_{r}\left(\sqrt{p(1-p)}n^{(r+1)/2}\right)$. Otherwise, since $\disc^{+}(F) = \disc^{+}(H) < 2^{-3r^2} \beta_r \sqrt{p(1-p)}n^{(r+1)/2} /\alpha$, Lemma \ref{lem:lemma 11} implies that 
    $$\disc^{-}(H) = \disc^{-}(F) \geq 2^{-3r^2} \beta_r \sqrt{p(1-p)}n^{(r+1)/2} \alpha.$$
    In particular, it follows that 
    $$\disc^{-}(H) \cdot \disc^{+}(H) = \Omega_r\left(p(1-p)n^{r+1}\right),$$
    which completes the proof. 
\end{proof}

\end{document}